\documentclass{article}
\newcommand{\spec}{\operatorname{Spec}}

\usepackage{amsmath,amssymb,amscd}
\usepackage{amsthm}
\newtheorem{thm}{\indent \sc Theorem}[section]
\newtheorem{prop}{\indent \sc Proposition}[section]
\newtheorem{cor}{\indent \sc Corollary}[section]
\newtheorem{lem}{\indent \sc Lemma}[section]
\newtheorem{defn}[thm]{Definition}

\newtheorem{rem}{\indent \sc Remark}[section]

\newtheorem{assump}[thm]{\indent \sc Assumption}
\usepackage[all]{xy}

\begin{document}
\title{Hasse principle for character group of finitely
generated field over the rational number field}
\author{Makoto Sakagaito}
\date{}
\renewcommand{\thefootnote}{}
\maketitle
\begin{abstract}
In this paper, we show the Hasse principle for the character group of a finitely
generated field over the rational number field. By applying this result, we obtain an algebraic proof of unramified class field theory of
arithmetical schemes.
\end{abstract}
\section{Introduction}
The Hasse principle for the character group of a global field is known as a
classical result (\cite[p.180, 8.8 Corollary]{C-F}). The classical class field theory is used to prove this result, especially the first inequality
(\cite[p.179, 8.4 Consequence]{C-F}).
The objective of this paper is to show the Hasse principle for the character group of a finitely
generated field over the rational number field $\mathbf{Q}$, that is, to show the following results.
\begin{thm}
\upshape
(Theorem \ref{l-g-c})
\begin{enumerate}
\item
Let $k$ be a finitely generated field over $\mathbf{Q}$; $X$, a regular algebraic curve over $k$;
     and $R(X)$, the function field of $X$. Let
     $\widetilde{R(X)}_{\mathfrak p}$ be the function
     field of the Henselization of $\mathcal{O}_{X, \mathfrak p}$ for $\mathfrak p\in X\backslash(0)$. Then, the
     local-global map of the character groups
\begin{equation*}
\operatorname{H}^{1}\left(R(X), \mathbf{Q}/\mathbf{Z}\right)\to
 \prod_{\mathfrak p\in X\backslash
 (0)}\operatorname{H}^{1}\left(\widetilde{R(X)}_{\mathfrak p}, \mathbf{Q}/\mathbf{Z}\right)
\end{equation*}
is injective.
\item
Let the Kronecker dimension of $k$ be $n$. For a certain set
     of $n$-dimensional local fields $\{k_{n, i}\}_{i\in P_{n}}$,
the local-global map
\begin{equation*}
\operatorname{H}^{1}(k, \mathbf{Q}/\mathbf{Z})\to \prod_{i\in P_{n}}
\operatorname{H}^{1}(k_{n, i}, \mathbf{Q}/\mathbf{Z})
\end{equation*}
is injective.
Moreover,
\begin{equation*}
\operatorname{H}^{1}(R(X), \mathbf{Q}/\mathbf{Z})\to \prod_{i\in P_{n}} \operatorname{H}^{1}(R(X\times
 k_{n, i}), \mathbf{Q}/\mathbf{Z})
\end{equation*}
is injective.
\end{enumerate}
\end{thm}
By applying the above results, we obtain the following:
\begin{cor}
\upshape
(Corollary \ref{cap})
Let $\mathcal{X}$ be a connected normal scheme of finite type over
 $\spec(\mathbf{Z})$ where the characteristic of $\mathcal{X}$ is $0$.
Suppose that  $i\le\operatorname{dim}(\mathcal{X})$. Then,
\begin{equation*}
\operatorname{H}^{1}(\mathcal{X}, \mathbf{Q}/\mathbf{Z})\to
\prod_{\mathfrak p\in\mathcal{X}_{(i)}}
\operatorname{H}^{1}(\kappa(\mathfrak p), \mathbf{Q}/\mathbf{Z})
\end{equation*}
is injective.
\end{cor}
This result yields an algebraic proof of the unramified class field theory of
arithmetical schemes (\cite[p.270, Theorem (5.10)]{S2}), which is proved
by using Chebotarev's density theorem in \cite{S2}.

\section{Notation}
For a scheme $X$, $X_{(i)}$ is the set of points of codimension $i$, and
$X^{(i)}$ is the set of points of dimension $i$. We denote by
$\pi_{1}^{ab}(X)$ the abelian fundamental group of $X$. For an integer
$m>0$, we identify the etale cohomology group $\operatorname{H}^{1}(X,
\mathbf{Z}/m\mathbf{Z})$ with the group of all continuous homomorphisms
$\pi_{1}^{ab}(X)\to \mathbf{Z}/m\mathbf{Z}$.
 
For an integral scheme $X$ and $\mathfrak p\in X$, let
 $\kappa(\mathfrak p)$ be the residue field at $\mathfrak p$; $R(X)$, the function
field of $X$; $\mathcal{O}_{X,
 \mathfrak p}$, the local ring at $\mathfrak p$;
 $\widetilde{\mathcal{O}_{X, \mathfrak p}}$, the Henselization of $\mathcal{O}_{X,
 \mathfrak p}$;
 $\widetilde{R(X)}_{\mathfrak p}$, its quotient field; $\widehat{\mathcal{O}_{X, \mathfrak p}}$, the completion of $\mathcal{O}_{X,
 \mathfrak p}$; $\widehat{R(X)}_{\mathfrak p}$, its quotient field;
 $\mathcal{O}_{X, \bar{\mathfrak p}}$, the strict Henselization of $\mathcal{O}_{X,
 \mathfrak p}$; and $R(X)_{\bar{\mathfrak p}}$, its quotient field.
When an integer $m$ is invertible in $\mathcal{O}_{X}$, $\mu_{m}$ denotes
the sheaf of $m$ -th roots of unity on the etale site $X_{et}$ of $X$. For a local ring $A$,
\begin{math}
\operatorname{K}_{n}^{M}(A)
\end{math}
denotes the Milnor $K$-group of degree $n$ for $A$.
\section{Higher local class field theory and its application}
In this section, our objective is to prove Proposition \ref{hlg}, which is a
generalization of \cite[p.524,  Lemma 5.4]{JS};
it is required to prove Theorem \ref{l-g-c}.
\subsection{Higher local field theory}
In this subsection, we review higher local field theory, which plays an important role in the proof of the main
result.
\begin{defn}
\label{DH}
\upshape
A field $K$ is called an $n$-\emph{dimensional local field} if there is a
sequence of fields $k_{n}, \cdots, k_{0}$ satisfying the following
conditions:
$k_{0}$ is a finite field, $k_{i}$ is a \emph{Henselian} discrete valuation
field with residue field $k_{i-1}$ for $i=1, 2, \cdots, n$, and $k_{n}=K$.
\end{defn}
For any field $k$ of characteristic $0$, and $r\geq 1$, let
\begin{equation*}
\operatorname{H}^{r}(k)=
\lim_{\to}\operatorname{H}^{r}(k, \mu_{m}^{\otimes(r-1)})
\end{equation*}
be the Galois cohomology groups.

Then, there exists a canonical isomorphism
\begin{equation*}
\eta: \operatorname{H}^{n+1}(K)\simeq \mathbf{Q}/\mathbf{Z}.
\end{equation*}
For $0\le r\le n+1$, the canonical pairing
\begin{equation*}
\langle~, ~\rangle :
 \operatorname{H}^{r}(K)\times\operatorname{K}_{n+1-r}^{M}(K)
\to \operatorname{H}^{n+1}(K)
\end{equation*}
induces a homomorphism
\begin{equation*}
\Phi_{K}^{r}: \operatorname{H}^{r}(K)\to
\operatorname{Hom}\left(\operatorname{K}_{n+1-r}^{M}(K), \mathbf{Q}/\mathbf{Z}\right)
\end{equation*}
\begin{thm}
\upshape
\label{hlt}
\cite{K1, K3}
Let $K$ be an $n$-dimensional local field and $F$ be the $n-1$-dimensional
 local field that is the residue field of $K$. Then,
\begin{enumerate}
\item
The correspondence
\begin{equation*}
L\to \operatorname{N}_{L/K}\operatorname{K}_{n}^{M}(L)
\end{equation*}
is a bijection from the set of all finite abelian extensions of $K$ to
     the set of all open subgroups of $\operatorname{K}_{n}^{M}(K)$ of
     finite index.
\item
$\Phi^{r}_{K}$ induces an isomorphism between $\operatorname{H}^{r}(K)$
     and the group of all
 continuous characters of finite order of $\operatorname{K}_{n+1-r}^{M}(K)$
 when $0\leqq r\leqq n+1$.
\item We have the commutative diagram
\begin{equation*}
\xymatrix{
\operatorname{K}_{n}^{M}(K)\ar[r]\ar[d]_{\partial}
& \operatorname{Gal}(K^{ab}/K)\ar[d]\\
\operatorname{K}_{n-1}^{M}(F)\ar[r]
&\operatorname{Gal}(F^{ab}/F),
}
\end{equation*}
where the horizontal arrows come from the class field theory, and the
      left vertical arrow $\partial$ is the boundary homomorphism.

In particular,
      an element $\chi\in \operatorname{H}^{1}(K,
      \mathbf{Q}/\mathbf{Z})$ is unramified, that is , the corresponding
      cyclic extension of $K$ is unramified, if and only if
      $\varPhi_{K}^{1}(\chi)$ is trivial on
      $\operatorname{Ker}(\partial)$.
\end{enumerate}
\end{thm}
%

\subsection{Class field theory of schemes over Henselian discrete valuation field}
Let $Z$ be an excellent scheme. For $x\in Z^{(i)}$ and $y\in Z^{(i+1)}$, let
\begin{equation*}
\partial_{x}^{y}: \operatorname{K}_{*+1}^{M}(\kappa(y))\to \operatorname{K}_{*}^{M}(\kappa(x))
\end{equation*}
be the following homomorphism. Let $W$ be the normalization of the
reduced scheme $\bar{\{y\}}$, and define $\partial_{x}^{y}$ as
\begin{equation*}
\partial_{x}^{y}=\sum_{x^{\prime}}\operatorname{N}_{\kappa(x^{\prime})/\kappa(x)}\circ
\partial_{x^{\prime}}.
\end{equation*}
Here, $x^{\prime}$ ranges over all points of $W$ lying over $x$,
$\partial_{x^{\prime}}$ denotes the tame symbol
\begin{equation*}
\operatorname{K}_{*+1}^{M}(\kappa(y))\to \operatorname{K}_{*}^{M}(\kappa(x^{\prime}))
\end{equation*}
associated with the discrete valuation ring
$\mathcal{O}_{Y, x^{\prime}}$ and
$\operatorname{N}_{\kappa(x^{\prime})/\kappa(x)}$ is
the norm map
\begin{equation*}
\operatorname{K}_{*}^{M}(\kappa(x^{\prime}))\to \operatorname{K}_{*}^{M}(\kappa(x)).
\end{equation*}
Then,
\begin{equation*}
\cdots\to \bigoplus_{x\in Z^{(2)}}\operatorname{K}_{n+2}^{M}(\kappa(x))\to
\bigoplus_{x\in Z^{(1)}}\operatorname{K}_{n+1}^{M}(\kappa(x))\to
\bigoplus_{x\in Z^{(0)}}\operatorname{K}_{n}^{M}(\kappa(x))
\end{equation*}
is complex by \cite{K2}.
For an integer $n$, we define
\begin{equation*}
\operatorname{SK}_{n}(Z)=\operatorname{Coker}
\left(
\partial: \bigoplus_{y\in Z^{(1)}}\operatorname{K}_{n+1}^{M}(\kappa(y))
\to \bigoplus_{x\in Z^{(0)}}\operatorname{K}_{n}^{M}(\kappa(x))
\right).
\end{equation*}
We consider schemes $\mathfrak X$, $X$, and $Y$ which satisfy the
following assumption.
\begin{assump}
\upshape
\label{ap}
\begin{enumerate}
\item[] $\mathcal{O}_{k}$ : a Henselian discrete valuation ring with
residue field $F$ and quotient field $k$,
\item[] $\mathfrak X$ : a connected regular proper
 scheme over $S=\spec(\mathcal{O}_{k})$,
\item[] $X=\mathfrak
 X\otimes_{\mathcal{O}_{k}}k$ and $Y=\mathfrak X\otimes_{\mathcal{O}_{k}}F$.
\end{enumerate}
\end{assump}
Then,
$X^{(1)}\subset\mathfrak{X}^{(2)}$, $\mathfrak{X}^{(1)}=X^{(0)}\cup Y^{(1)}$, and
$\mathfrak{X}^{(0)}=Y^{(0)}$.
Hence, we have the anti-commutative diagram
\begin{equation}
\xymatrix{
\bigoplus_{v\in X^{(1)}}\operatorname{K}_{n+1}^{M}(\kappa(v))\ar[r]\ar[d]
& \bigoplus_{u\in X^{(0)}}\operatorname{K}_{n}^{M}(\kappa(u))\ar[r]\ar[d]
& \operatorname{SK}_{n}(X)\ar[r]
& 0\\
\bigoplus_{y\in Y^{(1)}}\operatorname{K}_{n}^{M}(\kappa(y))\ar[r]
&\bigoplus_{x\in Y^{(0)}}\operatorname{K}_{n-1}^{M}(\kappa(x))\ar[r]
&\operatorname{SK}_{n-1}(Y)\ar[r]
& 0.
}
\end{equation}
Therefore, we obtain the homomorphism
\begin{equation*}
\operatorname{\partial}: \operatorname{SK}_{n}(X)\to \operatorname{SK}_{n-1}(Y).
\end{equation*}
Let $k$ be an $n$-dimensional local field. Suppose that $X$ is a smooth
proper algebraic curve over $k$. Then, we construct a homomorphism
\begin{equation*}
\tau: \operatorname{SK}_{n}(X)\to \pi^{ab}_{1}(X).
\end{equation*}
%
Let $K$ be the function field of $X$, $P=X^{(0)}$; $R_{\mathfrak p}$, the Henselization of
$\mathcal{O}_{X, \mathfrak p}$ for $\mathfrak p \in P$; and
$K_{\mathfrak p}$, the quotient field of $R_{\mathfrak p}$. Then, we define
\begin{equation*}
I_{K}=\textstyle\prod^{\prime}_{\mathfrak
 p\in P} \operatorname{K}_{n+1}^{M}(K_{\mathfrak p}),
\end{equation*}
where $\prod^{\prime}$ denotes the restricted product with respect to the subgroups
\begin{math}
\operatorname{Ker}(\partial_{\mathfrak
 p})=\operatorname{K}_{n+1}^{M}(R_{\mathfrak p})
\end{math}
for $\mathfrak p \in P$.

Let $\chi$ be an element of
$\operatorname{H}^{1}(K, \mathbf{Q}/\mathbf{Z})$ and
$\chi_{\mathfrak p}$ be the restriction of $\chi$ to
$\operatorname{H}^{1}(K_{\mathfrak p}, \mathbf{Q}/\mathbf{Z})$. Note
that $\chi_{\mathfrak p}$ is unramified for almost all $\mathfrak p\in
P$. Hence, if
$a=(a_{\mathfrak p})_{\mathfrak p\in P}$ is an element of $I_{K}$, by
Theorem \ref{rec} (iii) and the definition of $I_{K}$, we have
\begin{equation*}
\langle\chi_{\mathfrak p}, a_{\mathfrak p}\rangle_{\mathfrak p}
=\Phi_{K_{\mathfrak p}}^{1}(\chi_{\mathfrak p})(a_{\mathfrak p})
=0
\end{equation*}
for
almost all $\mathfrak p\in P$.

Consequently, we obtain a pairing
\begin{equation}
\label{pa1}
\langle~~\rangle_{K}: \operatorname{H}^{1}(K,
 \mathbf{Q}/\mathbf{Z})\otimes I_{K}\to \mathbf{Q}/\mathbf{Z},
\end{equation}
defined by
\begin{equation*}
\langle\chi, (a_{\mathfrak p})_{\mathfrak p\in P}\rangle_{K}
=\sum_{\mathfrak p\in P}\langle\chi_{\mathfrak p}, a_{\mathfrak
p}\rangle_{\mathfrak p}.
\end{equation*}
Then, we have the following result.
\begin{prop}
\label{rec}
\upshape
(c.f. \cite[p.57, II, Proposition 1.2]{S1})
Let $k$ be an $n$-dimensional local field. Let $X$ be a smooth proper curve over $k$ and let $K=R(X)$. Then, if $a=(a_{\mathfrak p})_{\mathfrak p\in P}$ is in the diagonal image of
$\operatorname{K}_{n+1}^{M}(K)$ in $I_{K}$, we have, for any
\begin{math}
\chi\in \operatorname{H}^{1}(K, \mathbf{Q}/\mathbf{Z}),
\end{math}
\begin{equation*}
\langle\chi, a\rangle_{K}=0.
\end{equation*}
\end{prop}
\begin{proof}
Let $p$ be a prime number. Then, it is sufficient to prove the statement
 for an element of
\begin{math}
\operatorname{H}^{1}(K, \mathbf{Z}/p\mathbf{Z}).
\end{math}
The proof is similar to \cite[p.57, II, Proposition 1.2]{S1}.
\end{proof}
By applying this result, we obtain a generalization of \cite[p.76, II, Theorem 7.1]{S1}.
\begin{cor}
\upshape
\label{almost2}
Let $L$ be a cyclic extension of $K$ in which almost all
 $\mathfrak p\in X^{(0)}$ split completely. Then, all $\mathfrak p\in X^{(0)}$ split
 completely in the extension.
\end{cor}
\begin{proof}
Now, we have a finite morphism
\begin{math}
X\to \mathbf{P}_{k}^{1}.
\end{math}
Every finite subset of $\mathbf{P}_{k}^{1}$ is contained in an open
 affine set, and so is $X$. Let $I$ be a subset of $X^{(0)}$ such that
 the elements of $X^{(0)} \backslash I$ split completely in $L$.
Then, there exists an open affine scheme $\spec(R)$ of $X$ that contains
 $I$.

For any positive integer $i$ and $\mathfrak p\in P$, let
\begin{math}
U^{i}(K_{\mathfrak p}^{*})
=\operatorname{Ker}\left(R_{\mathfrak p}^{*}
\to (R_{\mathfrak p}/\mathfrak p^{i}R_{\mathfrak p})^{*}\right)
\end{math}
and let
\begin{math}
U^{i}\operatorname{K}_{M}^{n+1}(K_{\mathfrak p})
\end{math}
be the subgroup generated by symbols $\{u, x_{1}, \cdots, x_{n}\}$ with
 $u\in U^{i}(K_{\mathfrak p})$ and $x_{1}, \cdots, x_{n}\in K_{\mathfrak p}^{*}$.

Then, the image
 of
\begin{equation*}
\operatorname{K}_{n+1}^{M}(K)\to
\prod_{\mathfrak p\in I}\operatorname{K}_{n+1}^{M}(K_{\mathfrak p})/
U^{i}\operatorname{K}_{n+1}^{M}(K_{\mathfrak p})
\end{equation*}
is surjective by the approximation theorem for a Dedekind domain $R$.

Moreover, a norm subgroup of $\operatorname{K}_{n+1}^{M}(K_{\mathfrak
 p})$ contains $U^{i}\operatorname{K}_{n+1}^{M}(K_{\mathfrak p})$ for
 a sufficiently large $i$.
Hence, the statement
 follows from Proposition \ref{rec} and Theorem \ref{hlt} (i), (ii).
\end{proof}
Let $C_{K}$ be the quotient of $I_{K}$ by the image of the subgroup
$\operatorname{K}_{n+1}^{M}(K)$.

By Proposition \ref{rec} and the pairing (\ref{pa1}), we obtain a pairing
\begin{equation}
\label{pa2}
\langle~~\rangle_{K}: \operatorname{H}^{1}(K, \mathbf{Q}/\mathbf{Z}) \otimes
C_{K}
\to \mathbf{Q}/\mathbf{Z}.
\end{equation}
Now, we observe that the quotient of
$C_{K}$ by the image of the subgroup
\begin{equation*}
\prod_{\mathfrak p\in P}\operatorname{K}_{n+1}^{M}(R_{\mathfrak p})
\end{equation*}
is canonically isomorphic to $\operatorname{SK}_{n}(X)$, because the
quotient of $\operatorname{K}_{n+1}^{M}(K_{\mathfrak p})$ by
$\operatorname{K}_{n+1}^{M}(R_{\mathfrak p})$ is isomorphic to
$\operatorname{K}_{n}^{M}(\kappa(\mathfrak p))$ via the boundary map
$\partial_{\mathfrak p}$.

Consequently, the pairing (\ref{pa2}) induces the pairing
\begin{math}
\operatorname{H}^{1}(K, \mathbf{Q}/\mathbf{Z})\otimes
 \operatorname{SK}_{n}(X)
\to \mathbf{Q}/\mathbf{Z},
\end{math}
and we obtain the homomorphism
\begin{math}
\tau: \operatorname{SK}_{n}(X)\to \operatorname{H}^{1}(K, \mathbf{Q}/\mathbf{Z})^{*}=\pi^{ab}_{1}(X).
\end{math}
\begin{cor}
\upshape
Let $k$ be an $n$-dimensional local field where the characteristic of $k$ is
 $0$ and $X$ is a smooth proper scheme over $k$. Then, we obtain the homomorphism
\begin{math}
\tau: \operatorname{SK}_{n}(X)\to \pi^{ab}_{1}(X).
\end{math}
\end{cor}
\begin{proof}
By the same argument as that in the proof of \cite[p.259, Lemma (3.2)]{S2},
 it is sufficient to prove the statement for a proper smooth curve over
 $k$. Thus, the proof is complete.
\end{proof}
Then, we have the following result for homomorphisms $\partial$ and $\tau$.
\begin{lem}
\label{delta}
\upshape
(c.f. \cite[p.261, Lemma 3.11 (2), (3)]{S2})
Let $k$ be an $n$-dimensional local field. Suppose that $X$ and $Y$
 satisfy Assumption \ref{ap}. Then,
\begin{enumerate}
\item The map $\partial$ is surjective.
\item We have the commutative diagram
\begin{equation}
\xymatrix{
\operatorname{SK}_{n}(X)\ar[d]_{\partial}\ar[r]^{\tau}
&\operatorname{\pi}^{ab}_{1}(X)\ar[d]^{\delta}
\\
\operatorname{SK}_{n-1}(Y)\ar[r]_{\tau}
&\operatorname{\pi}^{ab}_{1}(Y).
}
\end{equation}
Here, $\delta$ is the composition map
\begin{math}
\pi_{1}(X)\twoheadrightarrow\pi_{1}(\mathfrak X)\stackrel{\sim}{\gets}\pi_{1}(Y)
\end{math}
, in which the first map is surjective because $X$ is normal
  \cite[p.41, I, Examples 5.2 (b)]{Me} and the second map is an isomorphism
      by \cite[Theorem (3.1) and (3.4)]{A}.
\end{enumerate}
\end{lem}
\begin{proof}
The proof is similar to \cite[p.261, Lemma 3.11 (2),
 (3)]{S2}. We use Theorem \ref{hlt} (iii) to prove (ii).
\end{proof}
The following is known as a generalization of a Henselian regular local ring .
\begin{lem}
\label{lK}
\upshape
(\cite[p.263, Lemma 3.15]{S2})
Let $A$ be a Henselian regular local ring of dimension $\geqq 2$, with
 perfect residue field $F$ and quotient field $K$. Let $T$ be a
 regular parameter of $A$, and suppose that in case
 $\operatorname{ch}(K)=0$ and $\operatorname{ch}(K)=p> 0$, $(T)$ is the
 unique prime ideal of height one that divides $(p)$. Put
 $U=\spec(A[1/T])$.

Then, if
\begin{math}
\chi\in
 \operatorname{H}^{1}(U, \mathbf{Q}/\mathbf{Z})
\end{math}
induces an unramified
 character $\chi_{u}\in \operatorname{H}^{1}(u, \mathbf{Q}/\mathbf{Z})$
 for each $u\in U^{(0)}$, $\chi$ comes from
 $\operatorname{H}^{1}(\spec(A), \mathbf{Q}/\mathbf{Z})$.
\end{lem}
We show the following fact by using this result and higher local class
field theory.
\begin{prop}
\upshape
\label{sk}
(c.f. \cite[Proposition 3.12]{S2})
Let $k$ be an $n$-dimensional local field. Suppose that $\mathfrak{X}$, $X$, and $Y$ satisfy Assumption
 \ref{ap}. Moreover, $\mathfrak X$ is smooth over $S$. Let
\begin{math}
\chi\in\operatorname{H}^{1}(X, \mathbf{Q}/\mathbf{Z})
\end{math}
and
\begin{math}
\tilde{\chi}: \operatorname{SK}_{n}(X)\to \mathbf{Q}/\mathbf{Z}
\end{math}
be the induced homomorphism. Then, the following are equivalent:
\begin{enumerate}
\item
$\chi$ comes from the subgroup
\begin{equation*}
\operatorname{H}^{1}(Y, \mathbf{Q}/\mathbf{Z})
\simeq
 \operatorname{H}^{1}(\mathfrak X, \mathbf{Q}/\mathbf{Z})
\hookrightarrow\operatorname{H}^{1}(X, \mathbf{Q}/\mathbf{Z}).
\end{equation*}
\item $\tilde{\chi}$ factors through the map $\partial$.
\end{enumerate}
\end{prop}
\begin{proof}
 The proof is similar to that of \cite[Proposition 3.12]{S2}, as
 follows. (i) implies (ii) by Lemma \ref{delta}. Therefore, it is
 sufficient to show that (ii) implies (i). For $x\in Y^{(0)}$, let
$A_{x}$ be the Henselization of the local ring
 $\mathcal{O}_{x}$ of $\mathfrak{X}$ at $x$.
Let $U_{x}=\spec(A_{x})\times_{\mathfrak X}X$, and let
\begin{math}
\chi_{x}\in \operatorname{H}^{1}(U_{x}, \mathbf{Q}/\mathbf{Z})
\end{math}
be the restriction of $\chi$. Then, the canonical morphism $U_{x}\to X$
 induces a bijection
\begin{equation*}
j: (U_{x})^{(0)}\to \{u\in X^{(0)}| u\to x\},
\end{equation*}
and for $u\in (U_{x})^{(0)}$, $\kappa(j(u))\simeq\kappa(u)$.

We assume (ii). Then, for every $u\in (U_{x})^{(0)}$, the
 restriction
$\chi _{x, u}\in \operatorname{H}^{1}(u,
 \mathbf{Q}/\mathbf{Z})$ of $\chi_{x}$ corresponds to an unramified
 extension of $\kappa(u)$ by Theorem \ref{hlt} (iii). Let $t$ be a prime
 element of $\mathcal{O}_{k}$. Then, $t$ becomes a regular parameter of $A_{x}$ by the assumption
 of $X$. Hence, $\chi_{x}$ comes from $\operatorname{H}^{1}(\spec(A_{x}),
 \mathbf{Q}/\mathbf{Z})$ by Lemma \ref{lK}. Therefore, (ii) implies
 (i) by descent theory.
\end{proof}
\begin{defn}
\label{CS}
\upshape
Let $Z$ be a Noetherian scheme. A finite etale covering

\begin{math}
f: U\to Z
\end{math}
is called a \emph{c.s.~covering} if any closed point $x$ of $Z$ \emph{splits
 completely} in the covering, that is, $\spec\kappa(x)\times_{Z}U$ is
 isomorphic to a finite sum $\spec\kappa(x)$, where $\kappa(x)$ is the
 residue field of $x$. Let $\pi_{1}^{c.s.}(Z)$ denote the quotient
 group of $\pi_{1}^{ab}(Z)$ that classifies c.s. coverings of $Z$.
\end{defn}
Let $X$ be a normal scheme. Then,
\begin{equation*}
\operatorname{Hom}_{cont}\left(\pi_{1}^{c.s}(X), \mathbf{Q}/\mathbf{Z}\right)
=\operatorname{Ker}\left(\operatorname{H}^{1}(X, \mathbf{Q}/\mathbf{Z})
\to \prod_{\mathfrak p\in X^{(0)}}\operatorname{H}^{1}(\kappa(\mathfrak p),
 \mathbf{Q}/\mathbf{Z})\right)
\end{equation*}
by Definition \ref{CS}. In addition,
\begin{align*}
&\operatorname{Ker}\left(\operatorname{H}^{1}(X, \mathbf{Q}/\mathbf{Z})
\to \prod_{\mathfrak p\in X_{(1)}}\operatorname{H}^{1}(\kappa(\mathfrak p),
 \mathbf{Q}/\mathbf{Z})\right)\\
=&\operatorname{Ker}\left(\operatorname{H}^{1}(R(X), \mathbf{Q}/\mathbf{Z})
\to \prod_{\mathfrak p\in X_{(1)}}\operatorname{H}^{1}(\widetilde{R(X)}_{\mathfrak p},
 \mathbf{Q}/\mathbf{Z})
\right) .
\end{align*}
\begin{lem}
\upshape
\label{inh}
(c.f. \cite[p.524, Lemma 5.4]{JS})
Let $k$ be an $n$-dimensional local field. Suppose that $\mathfrak{X}$, $X$, and $Y$ satisfy Assumption \ref{ap}. Moreover,
 $\mathfrak{X}$ is smooth over $S$. Then, the specialization map
\begin{math}
\operatorname{\delta}_{\mathfrak X}
\end{math}
(where we omitted suitable base points) is surjective and induces an
 isomorphism
\begin{equation}
\label{c.s.}
\operatorname{\pi}_{1}(X)^{c.s.}\simeq
\operatorname{\pi}_{1}(Y)^{c.s.}.
\end{equation}
\end{lem}
\begin{proof}
The proof is similar to \cite[p.524, Lemma 5.4]{JS}. The
 surjectivity of $\delta_{X}$ follows from the definition. Therefore, the
 surjectivity of $(\ref{c.s.})$ follows. The injectivity of
 $(\ref{c.s.})$ follows from
Proposition \ref{sk}.
\end{proof}

Let $k_{0}, k_{1}, \cdots k_{n}=K$ be a sequence of fields as in
Definition \ref{DH}, and let $X$ be a proper algebraic curve
over $K$. Let $\mathcal{O}_{k_{i}}$ ($0\leqq i\leqq n$) be a Henselian
discrete valuation with quotient field $k_{i}$.
When $X$ satisfies the following condition, we say that $X$ has a good
reduction in each step.

$X$ has a good reduction, that is, there is a proper smooth morphism
$\mathfrak X\to \spec \mathcal{O}_{k_{n}}$. For any irreducible
component $X_{n-1}$ of $\mathfrak X\otimes_{\mathcal{O}_{k_{n}}}k_{n-1}$
, it has a good reduction. Moreover, we can repeat the above step until we obtain a proper smooth morphism $X_{0}\to \spec k_{0}$.

Then, we have the following result.
\begin{prop}
\label{hlg}
\upshape
Let $k$ be an $n$-dimensional local field. Let $X$ be a proper algebraic curve over $k$ and regular with a good
 reduction in each step, and let $P$ be
 the set that contains almost all closed points of $X$. Then,
\begin{equation*}
\operatorname{Ker}\left(\operatorname{H}^{1}(R(X), \mathbf{Q}/\mathbf{Z})
\to \prod_{\mathfrak p\in P}\operatorname{H}^{1}(\widetilde{R(X)}_{\mathfrak p},
 \mathbf{Q}/\mathbf{Z})
\right) =0.
\end{equation*}
\end{prop}
\begin{proof}
Proposition \ref{hlg} follows from Lemma \ref{inh} and Corollary \ref{almost2}.
\end{proof}
\section{Proof of Theorem \ref{l-g-c}}
First, we prove the following Lemmas in order to prove the main result.
\begin{lem}
\upshape
\label{gr}
Let $A$ be a Dedekind domain; $k$, its quotient field; and $X$, a proper
 algebraic curve over $k$.
Let $k_{\mathfrak p}$ be the Henselization of $k$ at $\mathfrak p \in \spec(A)$.
Suppose that the characteristic of $k$ is $0$.

Then, $X\times_{k} k_{\mathfrak p}$
has a good reduction for almost all $\mathfrak p \in \spec(A)$.
\end{lem}
\begin{proof}
Let $L$ be the function field of $X$. Let $Y$ be the normalization of
 $\mathbf{P}_{A}^{1}$ in $L$.
Then, $Y\to \mathbf{P}_{A}^{1}$ is finite (\cite[I, Proposition
 1.1]{Me}). Hence, $Y\to \mathbf{P}_{A}^{1}$ is proper (\cite[I, Proposition 1.4]{Me}). Therefore,
$f: Y\to \spec{A}$
is proper, and so is $f^{\prime}: Y\times_{\spec{A}}\spec k\to \spec k$,
 which is the base change of $f$ by $\spec k\to \spec A$.
Since $Y\times_{\spec{A}}\spec k$ is regular, $f^{\prime}$ is smooth and
$Y\times_{\spec{A}}\spec k\simeq X$ %
.

Let $T$ be the set of elements in $Y$ that is smooth over
 $\spec{A}$. Then, the proper map $f$ takes the complement of $T$ to a closed
 set $B$ of $\spec{A}$. Since $f^{\prime}$ is smooth, $B$ is a proper
 subset of $\spec{A}$, and hence, it consists of finite elements.

Let
 $A_{\mathfrak p}$ be the localization of $A$ at $\mathfrak
 p\in\spec{A}\backslash B$. Then,
\begin{math}
Y\times_{\spec{A}}
\spec{A_{\mathfrak p}}
\to\spec{A_{\mathfrak p}},
\end{math}
which is the base change of $f$ by $\spec A_{\mathfrak p}\to \spec A$,
is proper and smooth. Therefore, the statement follows.
\end{proof}
%

%
\begin{lem}
\upshape
\label{kgl}
Let $k$ be an arbitrary field with notations $A, X$ as in Lemma
 \ref{gr}. Let $P$ be the set that contains almost all closed points of
 $X$. Let $\operatorname{kgl}(P)$ denote the kernel of the local-global map
\begin{equation*}
\operatorname{H}^{1}(R(X), \mathbf{Q}/\mathbf{Z})\to
\displaystyle\prod_{x\in P}
\operatorname{H}^{1}(\widetilde{R(X)_{x}}, \mathbf{Q}/\mathbf{Z}) .
\end{equation*}
Suppose that there exists
 a separable algebraic field extension $k_{i}$ over $k$ for $i\in I$ and that the
 natural homomorphism
\begin{equation*}
\label{asumi}
\operatorname{H}^{1}(k, \mathbf{Q}/\mathbf{Z})\to
\displaystyle\prod_{i \in I}
\operatorname{H}^{1}(k_{i}, \mathbf{Q}/\mathbf{Z})
\end{equation*}
is injective. Let $P_{i}$ be the inverse image of $P$ by
 $X\times_{k}k_{i}\to X$.

Then, $P_{i}$ contains almost all closed points of $X\otimes_{k}k_{i}$ and
 the homomorphism
\begin{equation}
\label{kp}
\operatorname{kgl}(P)\to
\displaystyle\prod_{i \in I}\operatorname{kgl}
(P_{i})
\end{equation}
is injective.
\end{lem}
\begin{proof}
Let $K$ be a field; $K_{s}$, its separable closure; and $Z$, an algebraic
 curve over $K$.
Then,  $Z\times_{K} K_{s}\to Z$ is a Galois covering, and its Galois group is
$G(K_{s}/K)$. Then, there exists a Hochschild-Serre
spectral sequence
\begin{equation*}
\operatorname{H^{p}(G(K_{s}/K),
\operatorname{H}^{q}(Z\times_{K}K_{s}, \mathbf{Q}/\mathbf{Z}))\Rightarrow
\operatorname{H}^{n}(Z, \mathbf{Q}/\mathbf{Z})}
\end{equation*}
(cf, \cite[III, Theorem 2.20 and Remark 2.21]{Me}), and hence, we have an
 exact sequence
\begin{equation*}
0\to \operatorname{H}^{1}(K, \mathbf{Q}/\mathbf{Z})\to
\operatorname{H}^{1}(Z, \mathbf{Q}/\mathbf{Z})\to
\operatorname{H}^{1}(Z\times_{K}K_{s}, \mathbf{Q}/\mathbf{Z})^{G(K_{s}/K)}.
\end{equation*}
Therefore, the diagram
\begin{equation*}
\xymatrix{0\ar[r]
&\operatorname{H}^{1}(k,
 \mathbf{Q}/\mathbf{Z})\ar[r]\ar[d]
&\operatorname{H}^{1}(X,
 \mathbf{Q}/\mathbf{Z})\ar[r]\ar[d]
&\operatorname{H}^{1}(X\times_{k}k_{s},
 \mathbf{Q}/\mathbf{Z})^{G(k_{s}/k)}\ar[d]\\
0\ar[r]
&\displaystyle\prod_{i\in I}\operatorname{H}^{1}(k_{i},
\mathbf{Q}/\mathbf{Z})\ar[r]
&\displaystyle\prod_{i\in I}\operatorname{H}^{1}(X\times_{k}k_{i},
 \mathbf{Q}/\mathbf{Z})\ar[r]
&\displaystyle\prod_{i\in I}\operatorname{H}^{1}(X\times_{k}k_{s},
 \mathbf{Q}/\mathbf{Z})^{G(k_{s}/k_{i})}
}
\end{equation*}
is commutative. Then, the vertical map on the right-hand side is clearly
 injective, and so is that on the left-hand side by the assumption;
hence,
that of the middle,
\begin{equation}
\label{mgl}
\operatorname{H}^{1}(X, \mathbf{Q}/\mathbf{Z})\to
\prod_{i\in I}\operatorname{H}^{1}(X\times_{k}k_{i}, \mathbf{Q}/\mathbf{Z}),
\end{equation}
is injective. Since the homomorphism (\ref{kp}) is obtained by the
 commutative diagram
\begin{equation*}
\xymatrix{
\operatorname{H}^{1}(X, \mathbf{Q}/\mathbf{Z})
\ar[r]\ar[d]
&\operatorname{H}^{1}(X\times_{k}k_{i}, \mathbf{Q}/\mathbf{Z})\ar[d]\\
\displaystyle\prod_{y \in
 P}\operatorname{H}^{1}(\kappa(y),
 \mathbf{Q}/\mathbf{Z})\ar[r]
&\displaystyle\prod_{x \in P_{i}}
\operatorname{H}^{1}(\kappa(x), \mathbf{Q}/\mathbf{Z}),
}
\end{equation*}
the homomorphism (\ref{kp})
is injective.
\end{proof}
Let $F$ be a finitely generated field over $\mathbf{Q}$. Let the Kroneker
dimension of $F$ be
$n$. Then, there is a finitely generated field $F^{\prime}$ such that
$F^{\prime}$ is
algebraic closed in $F$
and the Kroneker dimension of $F^{\prime}$ is $n-1$.

Moreover, we have the regular
and proper algebraic curve $X$ over $F^{\prime}$ such that $R(X)=F$.
Hence, there exists an inclusion map $F\to F_{n, i}$, where $F_{n, i}$ is
an $n$-dimensional local field with a sequence of unramified extension
fields $k_{0}, \cdots k_{n}=F_{n, i}$ satisfying the conditions in
Definition \ref{DH}.

Then, the following main result follows from the above lemmas.
\begin{thm}
\upshape
\label{l-g-c}
\begin{enumerate}
\item
\label{m1}

Let $k$ be a finitely generated field over $\mathbf{Q}$. Let $X$ be
     a regular and proper algebraic curve over $k$. Let $P$ be the set
     that consists of almost all closed points of $X$.

Then, the local-global map of the character group
\begin{equation}
\label{HasseA}
\operatorname{H}^{1}\left(R(X), \mathbf{Q}/\mathbf{Z}\right)\to
\prod_{\mathfrak p\in
P}\operatorname{H}^{1}\left(\widetilde{R(X)}_{\mathfrak p}, \mathbf{Q}/\mathbf{Z}\right)
\end{equation}
is injective.
Hence, the homomorphism (\ref{HasseA}) is injective, regardless of whether $X$ is complete.
\item
\label{kn}

Moreover, let the Kronecker dimension of $k$ be $n$. Suppose that
the set
$\{k_{n,
     i}\}_{i\in P_{n}}$ is such that
an $n$-dimensional local field $k_{n, i}$ is derived from $k$ as shown above and
$X\times_{k}k_{n, i}$ has a good reduction in each step. Then, the local-global map
\begin{equation}
\label{hglh}
\operatorname{H}^{1}(k, \mathbf{Q}/\mathbf{Z})\to \prod_{i\in P_{n}}
\operatorname{H}^{1}(k_{n, i}, \mathbf{Q}/\mathbf{Z})
\end{equation}
is injective.
Moreover,
\begin{equation}
\label{HasseB}
\operatorname{H}^{1}(R(X), \mathbf{Q}/\mathbf{Z})\to \prod_{i\in P_{n}} \operatorname{H}^{1}(R(X\times
 k_{n, i}), \mathbf{Q}/\mathbf{Z})
\end{equation}
is injective.
\end{enumerate}
\end{thm}
\begin{proof}
\upshape
We prove the statement by induction. Suppose that $k$ is an
 algebraic number field and $X$ is the ring of integers of $k$. Then, the homomorphism
 (\ref{HasseA}) is injective by \cite[p.180, 8.8 Corollary]{C-F}. When (i)
 holds for  $k_{n}$ with the Kronecker dimension $n\ge 1$ , that is, the
 homomorphism (\ref{HasseA}) is injective, we see that
 (\ref{hglh}) is injective by Lemma \ref{gr}. Let $Z$ be a regular scheme. Then,
\begin{equation}
\label{etx}
0\to
\operatorname{H}^{1}(Z, \mathbf{Q}/\mathbf{Z})\to
\operatorname{H}^{1}(R(Z), \mathbf{Q}/\mathbf{Z})\to
\displaystyle\prod_{\mathfrak p\in Z_{(1)}}
\operatorname{H}^{1}(R(\spec(\mathcal{O}_{Z, \bar{\mathfrak p}})), \mathbf{Q}/\mathbf{Z})
\end{equation}
is an exact sequence. Since the sequence (\ref{etx}) is exact and the
 homomorphism (\ref{mgl}) in the proof of Lemma \ref{kgl} is injective, we see that
 the homomorphism (\ref{HasseB}) is injective, that is, (ii) holds.

When (ii) holds for $k_{n}$, we see that (i) holds for $k_{n}$ as
 follows. Let $\{\widetilde{k_{i}}\}_{i\in I}$ be the set of $n$-dimensional local
 fields that satisfy the assumption of (ii). Then, $\{\widetilde{k_{i}}\}_{i\in I}$ satisfy
 the assumption of Lemma \ref{kgl} by (\ref{hglh}). Hence, it is
 sufficient to show that $\operatorname{kgl}(P_{i})=0$ to show (i). It follows from Proposition
 \ref{hlg}. Therefore, the proof is complete.
\end{proof}
\begin{rem}
\upshape
A method of algebraic geometry enables us to prove Theorem \ref{l-g-c} by
 induction. Theorem \ref{l-g-c}
 follows from the Hasse principle for the character group of a global
 field, and it is proved by an algebraic method (c.f. \cite[p.180, 8.8 Corollary]{C-F}). Therefore, Theorem
 \ref{l-g-c} is proved by an algebraic method.
\end{rem}
By applying the main result, we obtain the following result.
\begin{cor}
\label{cap}
\upshape
Let $\mathcal{X}$ be a connected normal scheme of finite type over
 $\spec(\mathbf{Z})$ where the characteristic of $\mathcal{X}$ is $0$.
 Suppose that $i\le\operatorname{dim}(\mathcal{X})$. Then,
\begin{equation}
\label{ahp}
\operatorname{H}^{1}(\mathcal{X}, \mathbf{Q}/\mathbf{Z})\to
\prod_{\mathfrak p\in\mathcal{X}_{(i)}}
\operatorname{H}^{1}(\kappa(\mathfrak p), \mathbf{Q}/\mathbf{Z})
\end{equation}
is injective. The statement for $i=\operatorname{dim}(\mathcal{X})$ is
 equivalent to the following:

If $\mathcal{Y}\to \mathcal{X}$ is a connected
 c.s. covering, $\mathcal{Y}\to \mathcal{X}$ is an isomorphism.
\end{cor}
\begin{proof}
Let $\spec(A)$ be an open affine scheme of $\mathcal{X}$.
Let $\operatorname{f}:\mathbf{Z}\to A$ be a ring homomorphism that
 corresponds to a morphism of schemes
\begin{equation*}
\spec(A)\subset\mathcal{X}\to\spec(\mathbf{Z}).
\end{equation*}
Suppose that $R(\mathcal{X})$ is an algebraic function field in one
 variable over $K$. Then, we have a normal ring $B$ with $R(B)=K$ and
 $\operatorname{g}: B\to A$, which is an extension of
 $\operatorname{f}$. Since $\spec(A)\otimes_{\spec(B)}K$ is normal,
 Theorem \ref{l-g-c} holds in this case, and the homomorphism (\ref{ahp})
 is injective for $i=1$.

Moreover, let $A^{\mathfrak p}$ be the
 normalization of $A/\mathfrak p$. Since
\begin{math}
\spec(A^{\mathfrak p})\to\spec(A/\mathfrak p)
\end{math}
is finite by \cite[Proposition (7.8.6)]{EGA}, $A^{\mathfrak p}$ is a
 normal scheme of finite type over $\mathbf{Z}$. In addition, the homomorphism
\begin{math}
\operatorname{H}^{1}(\mathcal{X}, \mathbf{Q}/\mathbf{Z})\to
\operatorname{H}^{1}(\kappa(\mathfrak p), \mathbf{Q}/\mathbf{Z})
\end{math}
 goes through $\operatorname{H}^{1}(\spec A^{\mathfrak p},
 \mathbf{Q}/\mathbf{Z})$, and
\begin{math}
\operatorname{H}^{1}(\spec A^{\mathfrak p}, \mathbf{Q}/\mathbf{Z})
\to
\operatorname{H}^{1}(\kappa(\mathfrak p), \mathbf{Q}/\mathbf{Z})
\end{math}
is injective. Now, Corollary
 \ref{cap} holds for $\spec A^{\mathfrak p}$ and $i=1$. Therefore, it holds for
 $\mathcal{X}$ and $i=2$. By repeating the above argument,
 we can see that the statement holds for any $i\le \operatorname{dim}(\mathcal{X})$.
\end{proof}


\begin{thebibliography}{99}
\bibitem{A}
\textsc{M. Artin},
\textit{Algebraic approximation of structures over complete local
              rings},
Inst. Hautes \'Etudes Sci. Publ. Math.,
\textbf{36},
(1969),
23--58.
\bibitem{C-F}
\textsc{J.W.S. Cassels and A. Fr\"{o}hlich},
\textit{Algebraic number theory},
Academic Press,
London,
New York,
1967.
\bibitem{EGA}
\textsc{A. Grothendieck},
\textit{\'{E}l\'ements de g\'eom\'etrie alg\'ebrique. {IV}. \'{E}tude
              locale des sch\'emas et des morphismes de sch\'emas. {II}},
Inst. Hautes \'Etudes Sci. Publ. Math.
\textbf{24}
 (1965).
\bibitem{Me}
\textsc{J.~Milne},
\textit{\'{E}tale Cohomology},
Princeton Univ. Press,
Princeton,
1980.
\bibitem{JS}
\textsc{U.~Jannsen and S.~Saito},
\textit{Kato homology of arithmetic schemes
and higher class field theory over local fields},
Doc. Math. 2003,
Extra Vol.,
479-538 (electronic).
\bibitem{K1}
\textsc{K. Kato},
\textit{A generalization of local class field theory by using
	    {$K$}-groups, {II}},
J. Fac. Sci. Univ. Tokyo Sect. IA Math. \textbf{27}
(1980),
603--683.
\bibitem{K2}
\textsc{K. Kato},
\textit{Milnor {$K$}-theory and the Chow group of zero cycles}
(Boulder, CO, 1983),
Contemp. Math. \textbf{55},
Amer. Math. Soc.,
Providence,
1986,
pp. 241--253 .
\bibitem{K3}
\textsc{K. Kato},
Existence theorem for higher local fields, in
\textit{Invitation to Higher Local Fields},
eds. I. Fesenko and M. Kurihara,
Geometry and Topology Monographs, Vol. 3 (University of Warwick, 2000),
pp. 165--195.

\bibitem{S1}
\textsc{S. Saito},
\textit{Class field theory for curves over local fields},
Journal of Number Theory, (1) \textbf{21} (1985),
44--80.
\bibitem{S2}
\textsc{S. Saito},
\textit{Unramified class field theory of arithmetical schemes},
Ann. of Math. (2) \textbf{121}(1985),
251--281.
\end{thebibliography}
\end{document}